\documentclass[a4paper]{smfart}
\usepackage[francais,english]{babel}
\usepackage[latin1]{inputenc}
\usepackage[T1]{fontenc}
\DeclareFontFamily{T1}{wncyr}{}
\DeclareFontShape{T1}{wncyr}{m}{n}{%
  <5><6><7><8><9>gen*wncyr%
  <10><10.95><12><14.4><17.28><20.74><24.88>wncyr10}{}

\usepackage{calc}	    	
\usepackage{multicol}		
\usepackage{amssymb}	
\usepackage{wasysym}
\usepackage{verbatim}	
\usepackage{mathrsfs}	
\usepackage{dsfont}
\usepackage[dvipsnames]{xcolor}
\usepackage{pgf}
\usepackage{pgfpages}
\usepackage{tikz}
\usetikzlibrary{trees,shapes,%
  matrix,arrows,decorations.pathmorphing,backgrounds,positioning,fit}
\usepackage{indentfirst}
\usepackage{cancel}
\usepackage{fixltx2e} 
\usepackage{hyperref}

\theoremstyle{plain}

\newtheorem{thm}{Théorème}[section]
\newtheorem{lem}[thm]{Lemme}

\newtheorem{coro}[thm]{Corollaire}
\newtheorem*{thma}{Théorème}

\theoremstyle{definition}

\newtheorem{exm}[thm]{Exemple}

\theoremstyle{remark}
\newtheorem{rem}[thm]{Remarque}

\newcommand{\on}{\operatorname}

\newcommand{\mc}{\mathcal}
\newcommand{\mb}{\mathbf}

\newcommand{\mf}{\mathfrak}

\newcommand{\Z}{\mathbb{Z}}
\newcommand{\R}{\mathbb{R}}
\newcommand{\C}{\mathbb{C}}
\newcommand{\Q}{\mathbb{Q}}

\newcommand{\F}{\mathbb{F}}

\newcommand{\pp}{\mf p}

\newcommand{\m}{\mc{M}}

\newcommand{\ol}{\overline}

\newcommand{\mx}{\mbox}

\newcommand{\geqs}{\geqslant}
\newcommand{\leqs}{\leqslant}

\newcommand{\w}{\wedge}

\newcommand{\sm}{\setminus}

\newcommand{\dd}{\mathop{}\mathopen{}\mathrm{d}}

\newcommand{\lra}{\longrightarrow}

\newcommand{\p}{\mathbb{P}}

\newcommand{\Sp}{\on{Spec}}
\newcommand{\Spec}{\Sp}

\newcommand{\dme}[1][]{\mc{DM}_{gm}^{eff}}
\newcommand{\dm}[1][]{\mc{DM}_{gm}}

\newcounter{notemargin}
\tikzset{labf/.style={mathsc,yshift=-\eca}}
\tikzset{labfs/.style={mathss,yshift=-\eca}}
\tikzset{labr/.style={mathsc,yshift=\eca}}
\tikzset{labrs/.style={mathss,yshift=\eca}}
\tikzset{math mode/.style = {execute at begin node=$, execute at end node=$}}
\tikzset{mathscript mode/.style =%
 {execute at begin node=$\scriptstyle , execute at end node=$}}
\tikzset{math/.style = {execute at begin node=$, execute at end node=$}}
\tikzset{mathsc/.style =%
 {execute at begin node=$\scriptstyle , execute at end node=$}}
\tikzset{mathss/.style =%
 {execute at begin node=$\scriptscriptstyle , execute at end node=$}}
\tikzset{root/.style={draw,circle,inner sep=1pt,execute at begin node=$\bullet,
    execute at end node=$}}
\tikzset{roottest/.style={draw,circle,inner sep=#1pt}}
\tikzset{roots/.style={draw,circle,inner sep=2pt}}
\tikzset{bull/.style={fill,circle,minimum size=2pt,inner sep=0pt}}
\tikzset{leaf/.style={minimum size=2pt,inner sep=1pt}}
\tikzset{leafb/.style={minimum size=0pt,inner sep=0pt}}
\tikzset{intvertex/.style={mathsc,fill,circle,minimum size=0.6ex, inner sep=0pt}}
\tikzset{Reda/.style={-,double distance=0.3ex, draw=black}}
\tikzset{N/.style={-,thin,draw=black}}
\tikzset{Nd/.style={-,dotted, thin,draw=black}}

\title
{\'Equations fonctionnelles du dilogarithme}
\alttitle{Functional equations for Rogers dilogarithm}
\author{Ismael Soud{\`e}res}
\address{Universität Osnabrück \\Institut für Mathematik
\\
Albrechtstr. 28a \\
DE-49076 Osnabrück \\
Germany \\
ismael.souderes@uni-osnabrueck.de}
\date{\today}
\begin{document}
\thanks{}
\begin{abstract}
Cet article démontre une ``nouvelle'' famille d'équations fonctionnelles
\eqref{eq:Eqn} ($n\geqs 4$) satisfaites par le dilogarithme de Rogers. Ces
équations fonctionnelles reflètent la combinatoire des coordonnées diédrales des
espaces de modules de courbes de genres $0$, $\m_{0,n}$. Pour $n=4$, on retrouve
la relation de dualité et, pour $n=5$, la relation à $5$ termes du
dilogarithme. Dans une seconde partie, on démontre que la famille \eqref{eq:Eqn}
se réduit à la relation à $5$ termes. C'est, à la connaissance de l'auteur, la
première fois qu'une famille infinie d'équations fonctionnelles du dilogarithme
ayant un nombre croissant de variables ($n-3$ pour \eqref{eq:Eqn}) se réduit à
la relation à $5$ termes.

La réduction de cette famille d'équations à la relation de $5$-cycle explique
les guillemets de la première phrase.
 \end{abstract}

\begin{altabstract}
This paper proves a ``new'' family of functional equations \eqref{eq:Eqn} for
Rogers dilogarithm. These equations rely on the combinatorics of dihedral coordinates
on moduli spaces of curves of genus $0$, $\m_{0,n}$. For $n=4$ we find back the
duality relation while $n=5$ gives back the $5$ terms relation. It is then
proved that the whole family reduces to the $5$ terms relation. In the author's
knownledge, it is the first time that an infinite family of functional equations
for the dilogarithm with an increasing number of variables ($n-3$ for
\eqref{eq:Eqn}) is reduced to the $5$ terms relation.

This reduction explains the quotation marks around ``new'' at the beginning of
this abstract.
 \end{altabstract}
\thanks{La première partie de ce travail a été essentiellement effectuée pour un
  exposé lors de la conférence ``Numbers and Physics'' ; conférence d'ouverture
  du trimestre ``Multiple Zeta Values, Multiple Polylogarithms, and Quantum
  Field Theory'', ICMAT, Madrid, sept.-dec. 2014. L'auteur remercie les
  organisateurs pour leur invitation sans laquelle ce travail n'aurait pas
  commencé. Enfin, je tiens à remercier particulièrement H. Gangl pour son écoute, ses
  explications et ses encouragements.}
\maketitle
\tableofcontents
\section{Introduction}
La fonction dilogarithme $Li_2$ est une fonction classique de l'analyse connue
au moins depuis Euler. Plus récemment les travaux de D. Zagier
\cite{ZagierRemLi2,ZagierLi2GeomNum,ZagierLi2}  ont permis de mieux comprendre
comment la fonction $Li_2$ est au carrefour de la théorie des nombres et de la
géométrie moderne. Ces dernières années l'étude des systèmes $Y$ et  des algèbres
amassées a mis au jour de ``nouvelles équations fonctionnelles'' satisfaites
par la fonction $Li_2$ (voir par exemple \cite{ChapoFIRDCYS, KellerCTQDI,IIKKN1}.
Dans cet article nous prouvons une famille d'équations fonctionnelles
satisfaites par $Li_2$ en nous appuyant sur la géométrie des espaces de modules
de courbes en genre $0$.

Pour $z \in \C$ satisfaisant
$|z|<1$, la fonction dilogarithme $Li_2$ est définie  par 
\[
Li_2(z)=\sum_{k=1}^{\infty} \frac{z^k}{k^2}.
\] 
Elle généralise la fonction 
\[
Li_1(z)=-\log(1-z)=\sum_{k=1}^{\infty} \frac{z^k}{k} \qquad \mbox{pour } |z|<1
\] 
et satisfait l'équation différentielle
\[
\frac{d}{dz}Li_2(z)=-\frac{1}{z}\log(1-z)=\frac 1 z Li_1(z)
\]
qui induit, via une représentation intégrale, le prolongement analytique de
$Li_2$ sur $\C\sm [1;\infty [$. 

Le dilogarithme de Rogers est défini sur l'intervalle $]0,1[$ par 
\[
L(x)=Li_2(x)+\frac 1 2 \log(x)\log(1-x)
\]
étendu par $L(0)=0$ et $L(1)=\frac{\pi^2}{6}$. 
Le dilogarithme satisfait nombre d'équations fonctionnelles (voir par exemple
\cite{ZagierLi2}). Parmi les 
plus familières, on trouve la relation de dualité
\begin{equation}\label{eq:reflex}
L(x)+L(1-x)=L(1)
\end{equation} 
et la relation à $5$-termes
\begin{equation}\label{eq:5term}
L(x)+L(\frac{1-x}{1-xy})+L(\frac{1-y}{1-xy})+L(y)+L(1-xy)=3L(1).
\end{equation} 
Les deux équations ci-dessus sont valables a priori lorsque tous les arguments
sont compris entre $0$ et $1$. Pour une extension à $\R$ ou $\p^1(\R)$ (et à
valeur dans $\R/\frac{\pi^2}{2}\Z$, on consultera
\cite[Chap. II, section A ]{ZagierLi2}. 

Les symétries de ces deux relations correspondent aux symétries des espaces de
modules de courbes en genre $0$ : $\m_{0,4}$ et $\m_{0,5}$. On regardera donc ces
deux relations sur la ``cellule standard'' de $\m_{0,4}(\R)$ et $\m_{0,5}(\R)$
respectivement (\cite[section 2.7]{BrownMZVPMS}). F. Brown dans \cite{BrownMZVPMS} a introduit les
coordonnées diédrales $u_{i,j}=u_{j,i}$ sur les espaces de modules $\m_{0,n}$. Ces
coordonnées sont associées aux cordes strictes d'un polygone dont les sommets sont
numérotés par les éléments  de $\{1, \ldots, n\}$. Il est à noter que les indices
$i$ et $j$ font référence aux numéros des sommets.
L'ensemble des ``numérotations'' possibles des sommets, modulo l'action du groupe diédral
$D_{2n}$, définit un recouvrement affine de l'espace $\m_{0,n}$ (\cite[Section
2]{BrownMZVPMS}. La cellule standard est alors définie par le choix de la
numérotation cyclique des sommets de $1$ à $n$, c'est la structure diédrale
$\delta$ standard, et la condition $u_{i,j}>0$. On se place dans le cas de la structure
diédrale standard. On a en particulier :
\[
i,j \in \{1,\ldots , n\} \qquad \mx{ avec } i\neq j-1,j,j+1 \mod n.
\]
Ces coordonnées ne sont pas indépendantes et F. Brown a montré en particulier
que pour $i$ fixé, toute coordonnée $u_{i,j}$, comme toute fonction sur $\m_{0,n}$ s'écrit comme
une fonction rationnelle en les $u_{i,j}$ l'indice $j$ variant de $i+1$ à $i+n-1$
modulo $n$ (Corollaire 2.23 de \cite{BrownMZVPMS}). On montre ici
\begin{thma}[Theorème \ref{thm:eqfctn}]
Soit $n\geqs 4$ et $\chi_n$ l'ensemble des $n(n-3)/2$ paires $\{i,j\}$ avec
$1\leqs i,j\leqs n$ tel que $i$, $j$, $i+1$ et $j+1$ soient distincts modulo
$n$. Lorsque $u_{i,j}\leqs 0$ pour tout $i$ et $j$, on a 
\begin{equation}\label{eqintro:Eqn}
\sum_{\{i,j\} \in \chi_n} L(u_{i,j})=\frac{(n-3)(n-2)}{2}L(1).
\tag{Eq$_n$}
\end{equation}
\end{thma}   
L'équation \eqref{eqintro:Eqn} pour $n=5$ est exactement l'équation à $5$ termes
précédemment citée. On montre aussi 
\begin{thma}[Theorème \ref{thm:5->n}]
L'équation \eqref{eqintro:Eqn} pour $n=5$ implique l'équation
\eqref{eqintro:Eqn} pour tout $n \geqs 4$. 
\end{thma}

L'auteur espère par ailleurs que cette approche éclairera sous un autre angle
les équations fonctionnelles du dilogarithme associées à l'étude des algèbres
amassées et des $Y$-systèmes \cite{ChapoFIRDCYS,KellerCTQDI,IIKKN1,
  IIKKN2}. Plus précisemment, les équations \eqref{eqintro:Eqn} semblent proches
des équations 
fonctionnelles associées aux $Y$-systèmes de type $A_{n-3}$ comme présenté dans
\cite{ChapoFIRDCYS}. Pour $n=4$ (une
variable) et $n=5$ (deux variables), la fonction rationnelle
$f(t)=\dfrac{1}{1+\frac{1}{t}}=\dfrac{t}{1+t}$ avec $f(x)=u_{1,3}$ et
$f(y)=u_{1,4}$ permet de passer des $Y$-systèmes
\[
x, \, \frac{1}{x} \quad (n=4) \qquad \mbox{ et }\qquad 
x,\, y,\frac{1+x}{y},\, \frac{1+y}{x}, \, \frac{1+x+y}{xy} \quad (n=5)
\]
aux coordonnées dihédrales
\begin{multline*}
u_{1,3},\, u_{2,4}=1-u_{1,3} \quad (n=4) \qquad \mbox{ et }\\
u=u_{1,3},\,v=u_{1,4},\, u_{3,5}=\frac{1-v}{1-uv},\, u_{2,4}=\frac{1-u}{1-uv},\, 
 u_{2,5}=1-uv\quad (n=5).
\end{multline*} Cependant, dès
le cas $n=5$, les
deux types d'équations fonctionnelles diffèrent par leurs formes ; l'une faisant
intervenir $L(\frac{1}{1+x})$ et l'autre $L(u_{1,3})=L(\frac{x}{1+x})$ ; ainsi
que par leurs termes constants $2L(1)$ d'un côté et $3L(1)$ de l'autre. Ainsi  la
relation entre ces deux familles semblables d'équations fonctionelles n'est pas
immédiate. De plus, dès $n=6$,  la correspondance entre les termes des
$Y$-systèmes et les coordonnées dihédrales donnée par la fonction rationnelle
$f$ n'est plus évidente, laissant entrevoir la nécessité d'une fonction
rationnelle plus complexe. 

La section suivante rappelle certains aspects des coordonnées diédrales et leurs
comportements vis à vis des applications d'oubli de points marqués entre espaces
de modules de courbes. La section \ref{sec:eqfunct} prouve
le Théorème \ref{thm:eqfctn}. La  section \ref{sec:5->n} prouve le Théorème
\ref{thm:5->n} après quelques rappels sur les applications d'oublis de points
marqués entre espaces de modules $\m_{0,n}$.

\section{Coordonnées diédrales et espaces de modules $\m_{0,n}$}
\label{sec:coorddihedral}
\subsection{Coordonnées diédrales}
On rappelle ici brièvement la construction et quelques propriétés des coordonnées
diédrales sur $\m_{0,n}$. Pour plus de détails le lecteur consultera
\cite[Section 2]{BrownMZVPMS}. 

Soit $n \geqs 4$. On note $\m_{0,n}$ l'espace de module de courbes de genre $0$
à $n$ points marqués. Si $(\p^1)^n_*$ représente l'espace des $n$-uplets 
d'éléments distincts de $\p^1$, l'espace 
$\m_{0,n}$ s'identifie à $(\p^1)^n_*$ modulo l'action diagonale de
$\on{PSL}_2$.

Pour quatre éléments distincts $i$, $j$, $k$, $l$ de $\{1,\ldots,n\}$ le
birapport  $[ij|kl]$ est défini par
\[
[i\,j|k\, l]=\frac{(z_i-z_k)(z_j-z_l)}{(z_i-z_l)(z_j-z_k)}
\]
où  $z_1, \ldots, z_n$  sont les coordonnées affines standard sur
$(\p^1)^n$. Ces birapports sont $\on{PSL}_2$ invariants mais  ne sont bien évidemment pas
indépendants (par exemple $[i\,j|k\,l]=[lk|ji]$). Parmi ceux-ci on distingue les
\emph{coordonnées diédrales} définies pour $i,i+1,j,j+1$ distincts modulo $n$ par 
\[
u_{i,j}=[i\, i+1|j+1\, j]
\] 
qui ne dépend que de la paire $\{i,j\}$. On note $\chi_n$ l'ensemble de ces paires. 

Le nom \emph{diédral} prend sa source dans
la représentation suivante. Le groupe des permutations sur $n$ éléments $S_n$ agit
sur $\m_{0,n}$ par permutation des points marqués. Les points réels de $\m_{0,n}$
se décomposent en une union de composantes connexes. Pour l'une d'elle, la
cellule standard, les points marqués sont dans l' ''ordre cyclique'' naturel sur
le cercle $\p^1(\R)$: $z_1<z_2 <\cdots <z_n<z_1$. L'action du groupe diédral
préserve cette cellule ainsi que l'ensemble des birapports $u_{i,j}$. F. Brown
commence lui par fixer une structure diédrale sur les points marqués (c'est à
dire par choisir une composante connexe). Nous sommes ici dans un cas
particulier de sa construction. On note 
$\delta$, ou $\delta_n$ si la précision du nombre de points marqués est
nécessaire, l'ordre cyclique standard $\{1<2<\ldots <n<1\}$.

Les paires $\{i,j\}$ s'identifient aux cordes d'un $n$-gone dont les sommets sont
numérotés de $1$ à $n$. Les birapports $u_{i,j}$ satisfont les relations 
(voir \cite[Corollaire 2.3]{BrownMZVPMS})
\begin{equation}\label{eq:relcordes}
\prod_{\substack{ a \leqs i\leqs b-1 \\c \leqs j \leqs d-1}} u_{i,j} + 
\prod_{\substack{b\leqs k\leqs  c-1 \\ d\leqs l\leqs a-1}} u_{k,l}  =1
\end{equation}
où l'ensemble ordonné $\{1<\ldots<n\}$ a été identifié à $\Z/n\Z$ muni de
l'ordre cyclique standard $1<2<\ldots< n=0<1$.

L'identification entre les paires $\{i,j\}$ et les cordes d'un polygone est
représentée à la figure \ref{fig:cordeij}. Il faut noter que les côtés du
polygone portent le label correspondant au point marqué et que le sommet $i$
représente le couple $(z_i, z_{i+1})$.
\begin{figure}[h]
\begin{center}
\begin{tikzpicture}
\foreach \p in {0,1,2,3,4,5,6,7,8}
   \node[intvertex] (P\p)  at (360/8*\p:2){};
\draw (P7) -- (P0) node[midway, right, mathsc] {z_{j+1}}
           -- node[midway, right, mathsc] {z_j} (P1);
\draw (P3) -- (P4) node[midway, left, mathsc] {z_{i+1}}
           -- node[midway, left, mathsc] {z_{i}} (P5);
\draw[dashed] (P2)--(P3);
\draw[dashed] (P5)--(P6);
\draw[dashed] (P1)--(P2);
\draw[dashed] (P6)--(P7);
\draw (P4)-- (P0);
\node[mathsc,xshift=-2ex] at (P4) {i};
\node[mathsc,xshift=2ex] at (P0) {j};
\node[mathsc,xshift=-3ex] at (P3) {i+1};
\node[mathsc,xshift=3ex] at (P7) {j+1};
\node[mathsc,xshift=-3ex] at (P5) {i-1};
\node[mathsc,xshift=3ex] at (P1) {j-1};
\end{tikzpicture} 
\end{center}
\caption{}
\label{fig:cordeij}
\end{figure}
Les relations données par l'équation \eqref{eq:relcordes} sont représentées à la
figure \ref{fig:relcordes}.
\begin{figure}
\begin{center}
\begin{tikzpicture}
\foreach \p in {0,1,2,3,4,5,6,7,8}
{  \coordinate (C\p) at (360/8*\p:2){};
  \node[intvertex] (P\p)  at (360/8*\p:2){};
 } 
\foreach \p/ \n in {0/c-1,1/b,2/b-1,3/a,4/a-1,5/d,6/d-1,7/c}
{  \node[mathsc]  at (360/8*\p:2.35){\n};}
%
\draw (P1) -- (P2);
\draw (P3) -- (P4);
\draw (P5) -- (P6);
\draw (P7) -- (P8);
\draw[dashed] (P2)--(P3);
\draw[dashed] (P0)--(P1);
\draw[dashed] (P4)--(P5);
\draw[dashed] (P6)--(P7);
\node[mathsc] at (360/16*5:2.2){\mathbf A} ;
\node[mathsc]at (360/16*13:2.2){\mathbf A} ;
\node[mathsc]at (360/16*1:2.2){\mathbf B} ;
\node[mathsc]at (360/16*9:2.2){\mathbf B} ;
%
\filldraw[white,fill=gray!40,opacity=0.5] (C2) -- (C3) -- (C6) --  (C7)--  cycle;
\filldraw[white,fill=gray!40,opacity=0.5] (C0) -- (C1) -- (C4) --  (C5) --  cycle;
\end{tikzpicture} 
\end{center}
\caption{}
\label{fig:relcordes}
\end{figure}
Les indices des coordonnées $u_{i,j}$ du premier terme
de l'équation \eqref{eq:relcordes} sont dans la partie $\mathbf{A}$ pendant que
ceux du deuxième terme sont dans la partie $\mathbf{B}$ : les deux ensembles de
cordes se coupent complètement (\cite[p. 384 et Figure 2]{BrownMZVPMS}). Dans le
cas où $\mathbf{A}$ ne contient qu'une corde, la relation \eqref{eq:relcordes} s'écrit
\begin{equation}\label{eq:relijcorde}
u_{a,c}+\prod_{\substack{a+1\leqs k\leqs  c-1 \\ c+1\leqs l\leqs a-1}} u_{k,l}=1
\end{equation}
On note $I_n\subset \Z[u_{i,j}]$ l'idéal engendré par les équations
\eqref{eq:relcordes} pour $\{i,j\}\in \chi_n$. F. Brown définit l'espace
$\m_{0,n}^{\delta}$ par 
\[
\m_{0,n}^{\delta}=\Spec\left( \Z[u_{i,j}|\{i,j\}\in \chi_n]/I_n\right)
\]
et montre la suite d'inclusions 
\[
\m_{0,n} \lra \m_{0,n}^{\delta}\lra \ol{\m_{0,n}},
\]
ainsi que  (\cite[Lemme 2.5]{BrownMZVPMS})
\[
\m_{0,n}=\Spec\left( \Z[u_{i,j}, \frac{1}{u_{i,j}}|\{i,j\}\in \chi_n]/I_n\right) ;
\] 
$\ol{\m_{0,n}}$ désignant ci-desssus la compactification de Deligne-Mumford de $\m_{0,n}$.
Le cellule standard fermée $\Phi_n$  correspond aux points réels de
$\m_{0,n}^{\delta}$  pour lesquels $u_{i,j}\geqs 0$ pour toute paire
$\{i,j\}$, on a aussi (\cite[Section 2.7]{BrownMZVPMS})
\[
\Phi_n=\{0\leqs u_{i,j}\leqs 1 \mbox{ pour tout } \{i,j\}\in \chi_n \}\subset
\m_{0,n}^{\delta}(\R).
\]
On note $\Phi_n^{\circ}$ la cellule ouverte.

Si les coordonnées diédrales ne sont pas indépendantes, on peut cependant en
choisir des sous-ensembles donnant des coordonnées globales sur $\m_{0,n}$
(\cite[Corollaire 2.23]{BrownMZVPMS}). On travaille ici avec les sous-ensembles
``en étoile''
\[
\{u_{1,3}, \ldots , u_{1,n-1}\}.
\] 
Chaque birapport $[z_{i_1}z_{i_2}|z_{i_3}z_{i_4}]$ s'exprime ainsi comme une
fonction rationnelle en les ${u_{1,j}}$. Dans le  cas où $n=5$, les relations
\eqref{eq:relijcorde} donnent en particulier
\[
u_{2,5}=1-u_{1,3}u_{1,4},\quad u_{2,4}u_{2,5}=1-u_{1,3}, \quad u_{3,5}u_{2,5}=1-u_{1,4},
\]
ce qui ce traduit avec $x=u_{1,3}$ et $y=u_{1,4}$ par  
\[
u_{1,3}=x, \, u_{2,4}=\frac{1-x}{1-xy},\, u_{3,5}=\frac{1-y}{1-xy},\, 
u_{1,4}=y, u_{2,5}=1-xy.
\]
On retrouve  les $5$-termes de la relation de $5$-cycle de $L$.

\subsection{Applications d'oubli de points marqués}
Soit $J$ un sous-ensemble de $S=\{1,\ldots ,n\}$. On note $f_J$ le morphisme
\[
f_J : \m_{0,n} \longrightarrow \m_{0,n-|J|}
\]
oubliant les points marqués d'indice dans $J$. Le morphisme induit
$\m_{0,n}^{\delta} \longrightarrow \m_{0,n-|J|}^{\delta}$ est aussi noté
$f_J$. Les morphismes $f_{J_1}$ et $f_{J_2}$ commutent dès que $J_1\cap J_2
=\emptyset$ pourvu que $n-|J_1|-|J_2|\geqs 4$. La composition est alors égale à
$f_{J_1\cup J_2}$.  

Les coordonnées diédrales sont compatibles avec ces applications d'oubli de points
marqués. Il faut cependant être attentif à indexer les points marqués de
l'espace cible $\m_{0,n-|J|}$ par les éléments de 
\[
S\sm J=\{1,\ldots ,n\}\sm J.
\] 
Dans ce cas, le polygone correspondant aura des côtés décorés par $z_i$ avec $i
\in S \sm J$ et des sommets décorés par des uplets $(i,i+1,\ldots, i+k)$
traduisant les points oubliés $z_{i+1}$, ... $z_{i+k}$.
\begin{exm}
\begin{enumerate}
\item Pour $J=\{i\}$, on représente $f_J$ par 
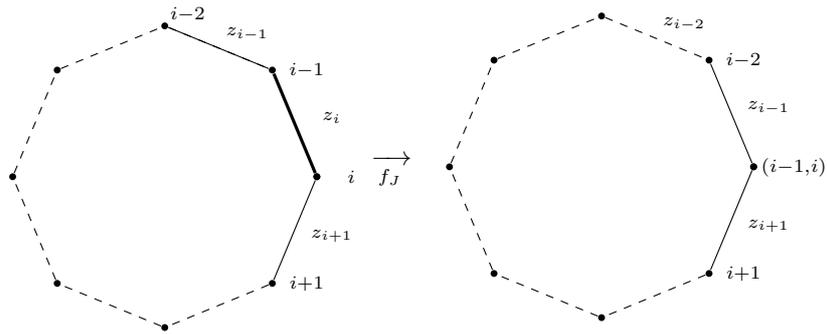
\[
\begin{tikzpicture}[baseline=(current bounding box.center)]
\foreach \p in {0,1,2,3,4,5,6,7,8}
{  \coordinate (C\p) at (360/8*\p:2){};
  \node[intvertex] (P\p)  at (360/8*\p:2){};
 } 
\node[mathsc,xshift=2ex,yshift=1ex] at (P2){i-2};
\node[mathsc,xshift=3ex] at (P1){i-1};
\node[mathsc,xshift=3ex] at (P0){i};
\node[mathsc,xshift=3ex] at (P7){i+1};
\node[mathsc,xshift=2ex] at (360/16:2.05){z_i};
\node[mathsc,xshift=2ex] at (-360/16:2.05){z_{i+1}};
\node[mathsc,xshift=2ex] at (360/16*3:2.05){z_{i-1}};
\draw[thin] (P2)--(P1);
\draw[thin] (P0)--(P7);
\draw[dashed] (P7)--(P6)--(P5)--(P4) -- (P3) -- (P2)--(P1);
\draw[very thick] (P1)--(P0);
\end{tikzpicture} 
\xrightarrow[f_J]{}\quad 
\begin{tikzpicture}[baseline=(current bounding box.center)]
\foreach \p in {0,1,2,3,4,5,6,7,8}
{  \coordinate (C\p) at (360/8*\p:2){};
  \node[intvertex] (P\p)  at (360/8*\p:2){};
 } 
\node[mathsc,xshift=3ex] at (P1){i-2};
\node[mathsc,xshift=3.5ex] at (P0){(i-1,i)};
\node[mathsc,xshift=3ex] at (P7){i+1};
\node[mathsc,xshift=2ex] at (360/16:2.05){z_{i-1}};
\node[mathsc,xshift=2ex] at (-360/16:2.05){z_{i+1}};
\node[mathsc,xshift=2ex] at (360/16*3:2.05){z_{i-2}};
\draw (P1)--(P0)--(P7);
\draw[dashed] (P7)--(P6)--(P5)--(P4) -- (P3) -- (P2)--(P1);
\end{tikzpicture} 
\]
\item Pour $J=\{3,4,7 \}\subset \{1,\ldots, 8\}$, on représente $f_J$ par
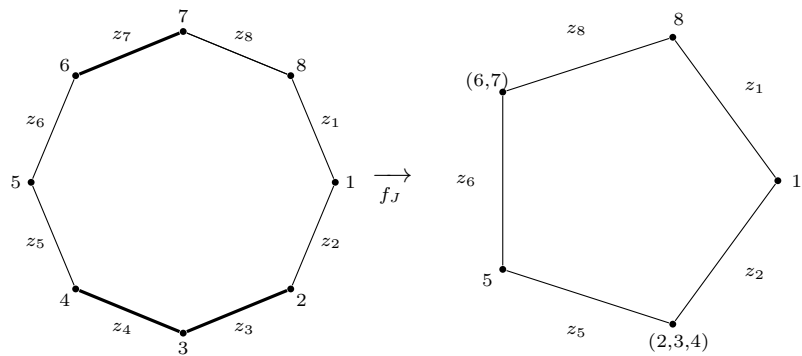
\[
\begin{tikzpicture}[baseline=(current bounding box.center)]
\foreach \p / \q \qq in {8/13/1,1/15/2,2/1/3,3/3/4,4/5/5,5/7/6,6/9/7,7/11/8}
{  \coordinate (C\p) at (360/8*\p:2){};
  \node[intvertex] (P\p)  at (360/8*\p:2){};
  \node[mathsc] at (-360/16*\q:2.1) {z_{\p}};
  \node[mathsc] at (-360/8*\p:2.2) {\qq};
 } 
\draw[thin] (P2)--(P1);
\draw[thin] (P0)--(P7);
\draw[very thick] (P7)--(P6);
\draw[very thick] (P6)--(P5);
\draw[thin] (P5)--(P4);
\draw[thin] (P4) -- (P3);
\draw[very thick] (P3) -- (P2);
\draw[thin] (P2)--(P1);
\draw[thin] (P1)--(P0);
\end{tikzpicture} 
\xrightarrow[f_J]{}\quad 
\begin{tikzpicture}[baseline=(current bounding box.center)]
\foreach \p in {0,1,2,3,4}
{  \coordinate (C\p) at (360/5*\p:2){};
  \node[intvertex] (P\p)  at (360/5*\p:2){};
 } 
\node[mathsc] at (360/5*2:2.25){(6,7)};
\node[mathsc] at (360/5*1:2.25){8};
\node[mathsc] at (360/5*0:2.25){1};
\node[mathsc] at (360/5*4:2.25){(2,3,4)};
\node[mathsc] at (360/5*3:2.25){5};
\node[mathsc] at (360/10*1:2.1){z_1};
\node[mathsc] at (360/10*3:2.1){z_{8}};
\node[mathsc] at (360/10*5:2.1){z_{6}};
\node[mathsc] at (360/10*7:2.1){z_{5}};
\node[mathsc] at (360/10*9:2.1){z_{2}};
\draw (P1)--(P0)--(P4)--(P3)--(P2)--(P1);
\end{tikzpicture} 
\]
\item On s'intéressera ici au cas où $J$ est constitué de suites successives d'un
  indice sur deux. Par exemple pour %
$J=\{4,6\}\cup\{10 \}\subset \{1,\ldots  10\}$ :
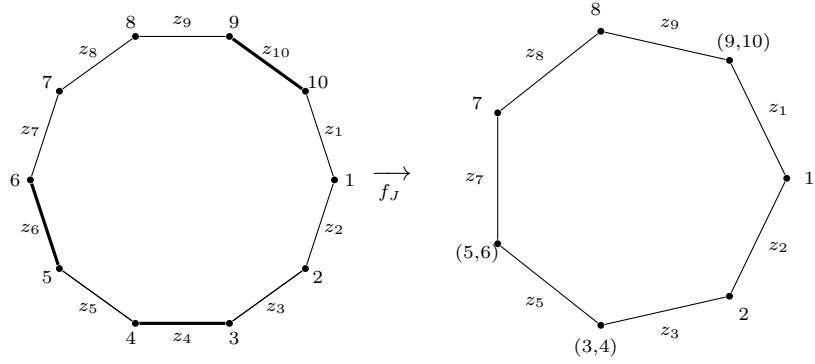
\[
\begin{tikzpicture}[baseline=(current bounding box.center)]
\foreach \p/\n/\pd/\nd in 
{0/1/1/2,
1/2/3/3,
2/3/5/4,
3/4/7/5,
4/5/9/6,
5/6/11/7, 
6/7/13/8,
7/8/15/9,
8/9/17/10, 
9/10/19/1}
{  \coordinate (C\p) at (-360/10*\p:2){};
  \node[intvertex] (P\p)  at (-360/10*\p:2){};
  \node[mathsc] at (-360/20*\pd:2.1) {z_{\nd}};
  \node[mathsc] at (-360/10*\p:2.2) {\n};
 } 
\draw[thin] (P2)--(P1);
\draw[thin] (P0)--(P9);
\draw[thin] (P8)--(P7);
\draw[thin] (P6)--(P7);
\draw[thin] (P5)--(P6);
\draw[thin] (P4) -- (P3);
\draw[very thick] (P2) -- (P3);
\draw[very thick] (P4)--(P5);
\draw[very thick] (P8)--(P9);
\draw[thin] (P3)--(P4);
\draw[thin] (P2)--(P1);
\draw[thin] (P1)--(P0);
\end{tikzpicture} 
\xrightarrow[f_J]{}\quad 
\begin{tikzpicture}[baseline=(current bounding box.center)]
\foreach \p/\n/\pd/\nd in 
{
0/1/1/2,
1/2/3/3,
2/{(3,4)}/5/5,
3/{(5,6)}/7/7,
4/7/9/8,
5/8/11/9, 
6/{(9,10)}/13/1
}
{  \coordinate (C\p) at (-360/7*\p:2){};
  \node[intvertex] (P\p)  at (-360/7*\p:2){};
  \node[mathsc] at (-360/14*\pd:2.1) {z_{\nd}};
  \node[mathsc] at (-360/7*\p:2.3) {\n};
 } 
\draw (P1)--(P0)--(P6)--(P5)--(P4)--(P3)--(P2)--(P1);
\end{tikzpicture} 
\]
Dans cet exemple, une des ``suites d'un indice sur deux'' est $(4,6)$, l'autre
simplement $10$. 
\end{enumerate}
\end{exm}

Le $n-|J|$ polygone correspondant à $\m_{0,n-|J|}$ mais décoré en tenant compte
du morphisme $f_J$ est noté $\Gamma_n^J$. Un sommet $\mb i=(i,i+1,\ldots , i+k)$
de ce polygone rencontre ainsi les côtés décorés  par $z_i$ et
$z_{i+k+1}$. Comme précédemment, à une corde stricte de $\Gamma_n^J$ correspond
un unique birapport
\[
u_{\mb i, \mb j}^{n,J}=[z_i\, z_{i+k+1}|z_{j+l+1}\,z_j] \in \mc O(\m_{0,n-|J|})
\]
où $\mb i=(i,\ldots , i+k)$ et $\mb j=(j,\ldots,j+l)$. Si $m=n-|J|$, on écrira   parfois
simplement $u_{\mb  i, \mb j}^{m}$ pour $u_{\mb i, \mb j}^{n,J}$. Lorsque $\mb
i=(i)$ (ou $\mb j=(j)$), on utilisera parfois la notation plus simple
$u^{n,J}_{i,\mb j}=u^m_{i,\mb j}$. Ces simplifications de la notation s'expliquent
par la compatibilité  des coordonnées diédrales vis à vis de $f_J$ 
(\cite[Lemma 2.9]{BrownMZVPMS}) 
\begin{equation}\label{eq:fT*uij}
f_T^*(u_{\mb i, \mb j}^{n, J})=\prod_{\substack{i \in \mb i\\ j \in \mb j}}u_{i,j}
\end{equation}
 où par un abus de notation $\mb i$ (resp. $\mb j$) représente l'ensemble
 associé au uplet $\mb i$ (resp. $\mb j$). 

\begin{rem}\label{rem:fT*uij}
L'équation ci-dessus montre en
 particulier que $f_T^*(u_{\mb i, \mb j}^{n, J})$ ne dépend que de $\mb i$ et de
 $\mb j$. Ceci justifie la simplification $u_{\mb i, \mb j}^{n, J}=u_{\mb i, \mb
   j}^{m}$ que l'on pourrait simplement noter $u_{\mb i, \mb j}$.
\end{rem}
\section{\'Equations fonctionnelles du dilogarithme}
\label{sec:eqfunct}
Après ces préliminaires, on montre ici que le dilogarithme de Rogers $L$
satisfait une équation fonctionnelle sur $\m_{0,n}$ pour chaque $n\geqs 4$. On
fixe $n\geqs 4$ et on note $\Phi_n^{\circ}$ la cellule standard ouverte de
$\m_{0,n}(\R)$. Cette cellule $\Phi_n^{\circ}$ peut s'identifier au simplexe
ouvert standard $\Delta_{n-3}=\{0<t_1<\cdots <t_{n-3}<1\}$ (coordonnées
simpliciales de $\m_{0,n}$) ou avec le cube $[0,1]^{n-3}$ (coordonnées cubiques
de $\m_{0,n}$). On utilisera ici les coordonnées diédrales présentées
précédemment. 
\begin{thm}\label{thm:eqfctn} Pour tout point $P$ de $\Phi_n^{\circ}$ on a 
\begin{equation}\label{eq:Eqn}
\sum_{\{i,j\} \in \chi_n} L(u_{i,j})=\frac{(n-3)(n-2)}{2}L(1)
\tag{Eq$_n$}
\end{equation}
où par un abus de notations $u_{i,j}=u_{i,j}(P)$.
\end{thm}
\begin{coro}\label{coro:Eqnf}
On repère $P\in \Phi_n^{\circ}$ par ses coordonnées $P=(u_{1,3},\ldots
u_{1,n-1})$ avec $u_{i,j}\in 
[0,1]$. On a alors:  
\begin{equation}\label{eq:Eqnf}
\sum_{\{i,j\} \in \chi_n} L(f_{i,j}(u_{1,3},\ldots, u_{1,n-1}))=\frac{(n-3)(n-2)}{2}L(1)
\end{equation}
où $f_{1,j}(u_{1,3},\ldots, u_{1,n-1})=u_{1,j}$ pour $j$ dans $\{3, n-1\}$ et où,
pour $i,j \neq 1$, $f_{i,j}$ est une fonction rationnelle en les $u_{1,j}$ donnée
par les relations de l'équation \eqref{eq:relcordes} (cf. \cite[Corollaire
2.23]{BrownMZVPMS}).  
\end{coro}
\begin{proof}[Preuve du Théorème \ref{thm:eqfctn}] On montre que la
  différentielle du membre de droite de \eqref{eq:Eqn} est nulle ; la constante
  étant déterminée par la limite pour $u_{1,3}=\cdots =u_{1,n-1}=0$ et les
  relations issues de \eqref{eq:relcordes}.

On rappelle que $L(x)=Li_2(x)+\frac 1 2 \log(x)\log(1-x)$. Pour une paire
$\{i,j\}$ de $\chi_n$ on a  donc
\[
\dd\left(L(u_{i,j})\right)=-\frac 1 2 \,
\frac{\dd u_{i,j}}{u_{i,j}}\log(1-u_{i,j})
+\frac 1 2\, \frac{\dd(1-u_{i,j})}{1-u_{i,j}}\log(u_{i,j}).
\]
Dans le cas d'une corde, la relation \eqref{eq:relcordes} se réduit à l'équation
\eqref{eq:relijcorde}:
\[
u_{i,j}+\prod_{\substack{i+1\leqs k\leqs  j-1 \\ j+1\leqs l\leqs i-1}} u_{k,l}=1
\]
où l'on identifie $\{1,\ldots,n\}$ avec $\Z/n\Z$ toujours muni de l'ordre
``cyclique standard''. 
On en déduit en particulier :
\[
\log(1-u_{i,j})=\sum_{\substack{i+1\leqs k\leqs  j-1 \\ j+1\leqs l\leqs i-1}} \log(u_{k,l})
\quad \mbox{ %
et }%
\quad 
\frac{\dd(1-u_{i,j})}{1-u_{i,j}}=
\sum_{\substack{i+1\leqs k\leqs  j-1 \\ j+1\leqs l\leqs i-1}} \frac{ \dd (u_{k,l})}{u_{k,l}}.
\]

En notant $E_n$ le membre de gauche de \eqref{eq:Eqn} on
a alors 
\begin{align*}
2\dd(E_n)=&\sum_{\{i,j\} \in \chi_n} 2\dd\left(L(u_{i,j})\right)\\
&=
\sum_{\{i,j\} \in \chi_n}\left(
-\frac{\dd u_{i,j}}{u_{i,j}}\log(1-u_{i,j})
+\frac{\dd(1-u_{i,j})}{1-u_{i,j}}\log(u_{i,j})\right)\\[1.5em]
=&-\sum_{\{i,j\} \in \chi_n}\left(
\frac{\dd u_{i,j}}{u_{i,j}}\log(1-u_{i,j})\right)
+\sum_{\{i,j\} \in \chi_n}\left(
\frac{\dd(1-u_{i,j})}{1-u_{i,j}}\log(u_{i,j})\right)\\[1.5em]
=&-\sum_{\{i,j\} \in \chi_n}\left(
\vphantom{\sum_{\substack{i+1\leqs k\leqs  j-1 \\ j+1\leqs l\leqs i-1}} }
\frac{\dd u_{i,j}}{u_{i,j}}\log(1-u_{i,j})\right)\\[1em]
&\hphantom{-\sum_{\{i,j\} \in \chi_n}\frac{\dd u_{i,j}}{u_{i,j}}\log()}
+\sum_{\{i,j\} \in \chi_n}\left(
\left(\sum_{\substack{i+1\leqs k\leqs  j-1 \\ j+1\leqs l\leqs i-1}} 
\frac{ \dd (u_{k,l})}{u_{k,l}}\right)
 \log(u_{i,j})\right)
\end{align*} 
En intervertissant les deux signes de somme de la dernière ligne on trouve 
\begin{multline*}
2\dd(E_n)=-\sum_{\{i,j\} \in \chi_n}\left(
\frac{\dd u_{i,j}}{u_{i,j}}\log(1-u_{i,j})\right)\\[1em]
+\sum_{\{k,l\} \in \chi_n}\left(
\frac{\dd (u_{k,l})}{u_{k,l}}\left(
\sum_{\substack{k+1\leqs i\leqs  l-1 \\ l+1\leqs j\leqs k-1}}
\log(u_{i,j})
\right) \right).
\end{multline*}
La somme des logarithmes du deuxième terme est égale à $\log(1-u_{k,l})$ d'où
\[
2\dd(E_n)=-\sum_{\{i,j\} \in \chi_n}\left(
\frac{\dd u_{i,j}}{u_{i,j}}\log(1-u_{i,j})\right)
+\sum_{\{k,l\} \in \chi_n}\left(
\frac{\dd (u_{k,l})}{u_{k,l}}\log(1-u_{k,l})\right)=0\,;
\]
ce qui conclut la preuve.
\end{proof}
En spécialisant l'une des coordonnées $u_{i,j}=0$, les relations
\eqref{eq:relcordes} se transforment en deux familles de relations du même type
; l'une des familles correspondant à $\m_{0,n_1}$, l'autre  à $\m_{0,n_2}$ avec
$n_1+n_2=n+2$. Cette décomposition est l'équivalent en coordonnées de l'inclusion
de  $\ol{\m_{0,n_1}}\times \ol{\m_{0,n_2}}$ dans $\ol{\m_{0,n}}$ (\cite[Section
2.3]{BrownMZVPMS}).  Ainsi, par
spécialisation, l'équation fonctionnelle \eqref{eq:Eqn} implique l'équation
(Eq$_{n-1}$). Dans le cas $n=5$, on obtient en particulier:
\begin{lem}\label{lem:Eq5->Eq4} En spécialisant $y=u_{1,4}=0$ l'équation
\[
L(u_{1,3})+L(u_{2,4})+L(u_{3,5})+L(u_{1,4})+L(u_{2,5})=3L(1)
\]
devient 
\[
L(u_{1,3})+L(u_{2,4})=L(1) \qquad \mbox{avec }u_{1,3}+u_{2,4}=1.
\]
On retrouve ici la relation de réflexion.
\end{lem} 
\begin{proof}
Les relations \eqref{eq:relcordes} sur $\m_{0,5}$ donnent en particulier
\begin{gather*}
u_{2,5}=1-u_{1,3}u_{1,4}, \qquad 
u_{3,5}=1-u_{1,4}u_{2,4}, \\
u_{1,3}+u_{2,4}u_{2,5}=1.
\end{gather*}
En spécialisant $u_{1,4}=0$, on trouve $u_{2,5}=u_{3,5}=1$ et
$u_{1,3}+u_{2,4}=1$. 
\end{proof}

\section{Réduction au cas de la relation à $5$-termes et de $\m_{0,5}$}\label{sec:5->n}
\begin{thm}\label{thm:5->n}
La famille d'équation \eqref{eq:Eqn} ($n\neq 5$) se déduit de l'équation
fonctionnelle à $5$-termes de $L$, \rm{(Eq$_5$)}.   
\end{thm}
On a vu que \rm{(Eq$_5$)} implique \rm{(Eq$_4$)}. Le cas $n=6$ étant un peu à
part car $6$ est ``trop petit'', on le traite ci-dessous. Pour $n=6$ on a  
\begin{multline*}
L(u_{1,3})+L(u_{2,4})+L(u_{3,5})+L(u_{4,6})+L(u_{1,5})+L(u_{2,6})+ \\
L(u_{1,4})+L(u_{2,5})+L(u_{3,6})=6L(1).
\end{multline*}
On note $E_6$ le membre de gauche de l'équation ci-dessus. En utilisant les
applications d'oubli $\m_{0,6}\lra \m_{0,5}$ pour $J=\{2\}$, $J=\{4\}$ et
$J=\{6\}$,  on montre ci-dessous que $E_6$ s'écrit comme une somme de termes de
type $E_5$ associés à l'équation
\rm{(Eq$_5$)}. Par exemple pour $J=\{2\}$, on a l'équation associée à
$\m_{0,5}$,
\[
\sum_{(\mb i,\mb j)\in \chi_5}L\left(u_{\mb i,\mb j}^{6,J}\right))=3L(1)
\] 
En composant chacune des fonctions ci-dessus avec $f_J:\m_{0,6}\lra \m_{0,5}$,
on obtient
\[
\sum_{(\mb i,\mb j)\in \chi_5}L\left(f_J^*(u_{\mb i,\mb j}^{6,J})\right))=3L(1).
\]
En utilisant l'équation \eqref{eq:fT*uij}, ceci se réécrit
\[
L(u_{3,5})+L(u_{4,6})+L(u_{1,5}u_{2,5})+ 
L(u_{1,4}u_{2,4})+L(u_{3,6})=3L(1).
\]
Visuellement le membre de droite se représente par le polygone associé à $f_J$
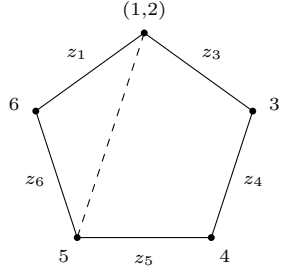
\[
\begin{tikzpicture}[baseline=(current bounding box.center)]
\foreach \p/\n/\pd/\nd in 
{
0/{(1,2)}/1/3,
1/3/3/4,
2/4/5/5,
3/5/7/6,
4/6/9/1
}
{ \coordinate (C\p) at (90-360/5*\p:1.5){};
  \node[intvertex] (P\p)  at (90-360/5*\p:1.5){};
  \node[mathsc] at (90-360/10*\pd:1.51) {z_{\nd}};
  \node[mathsc] at (90-360/5*\p:1.8) {\n};
 } 
\draw (C0)--(C1)--(C2)--(C3)--(C4)--cycle;
\draw[dashed] (C0)--(C3);
\end{tikzpicture} 
\]
étant entendu que l'on applique $L$ à l'image inverse des coordonnées données par les
cordes et que l'on somme chacun de ces termes.  Ces images inverses se lisent
directement sur le polygone ;  dans l'exemple
ci-dessus pour la corde $((1,2),5)$, on ``lit'' :
\[
f_J^*(u_{(1,2),5}^{6,J})=u_{1,5}u_{2,5}.
\]
On procède de même pour $J=\{4\}$ et $J=\{6\}$ obtenant ainsi trois équations
\rm(Eq$_5$), que l'on additionne pour obtenir
\begin{multline*}
E_6+L(u_{1,5}u_{2,5})+  L(u_{1,4}u_{2,4}) +L(u_{3,6}u_{4,6})+
L(u_{1,4}u_{1,3})+\\
L(u_{2,5}u_{2,6})+ L(u_{3,5}u_{3,6})= 9L(1).
\end{multline*}
Enfin les relations \eqref{eq:relcordes} pour $\m_{0,6}$ donnent 
\[
u_{1,4}u_{2,4}+u_{3,5}u_{3,6}=1,\quad u_{3,6}u_{4,6}+u_{1,5}u_{2,5}=1, 
\quad
u_{1,4}u_{1,3}+u_{2,5}u_{2,6}=1.
\]
On conclut alors le cas de \rm{Eq$_6$}, en utilisant la relation de
réflexion. Les relations \rm{Eq$_4$} et \rm{Eq$_5$} utilisées ici se résument par
les trois pentagones et les 
trois carrés ci-dessous (on a omis de noter les $z_i$) :
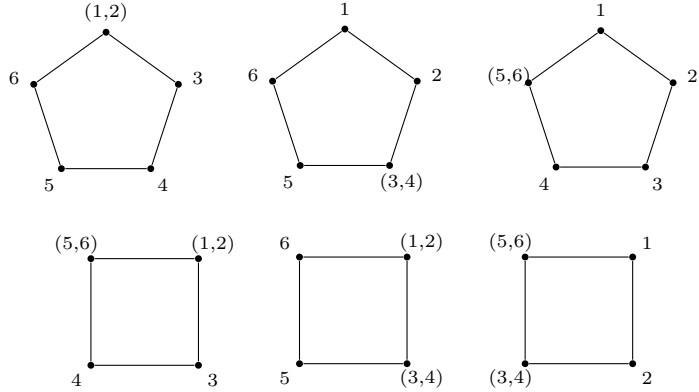
\[
 \begin{tikzpicture}[baseline=(current bounding box.center)]
 \foreach \p/\n in 
{
0/{(1,2)},
1/3,
2/4,
3/5,
4/6
 }
 { 
   \node[intvertex] (P\p)  at (90-360/5*\p:1.0){};
   \node[mathsc] at (90-360/5*\p:1.27) {\n};
  } 
 \draw (P0)--(P1)--(P2)--(P3)--(P4)--(P0);
 \end{tikzpicture} 
\quad
 \begin{tikzpicture}[baseline=(current bounding box.center)]
 \foreach \p/\n in 
{
0/1,
1/2,
2/{(3,4)},
3/5,
4/6
 }
 { 
   \node[intvertex] (P\p)  at (90-360/5*\p:1.0){};
   \node[mathsc] at (90-360/5*\p:1.27) {\n};
  } 
 \draw (P0)--(P1)--(P2)--(P3)--(P4)--(P0);
 \end{tikzpicture} 
\quad
\begin{tikzpicture}[baseline=(current bounding box.center)]
 \foreach \p/\n in 
{
0/1,
1/2,
2/3,
3/4,
4/{(5,6)}
 }
 { 
   \node[intvertex] (P\p)  at (90-360/5*\p:1.0){};
   \node[mathsc] at (90-360/5*\p:1.27) {\n};
  } 
 \draw (P0)--(P1)--(P2)--(P3)--(P4)--(P0);
 \end{tikzpicture} 
\]
\[
 \begin{tikzpicture}[baseline=(current bounding box.center)]
 \foreach \p/\n in 
{
0/{(1,2)},
1/3,
2/4,
3/{(5,6)}
 }
 { 
   \node[intvertex] (P\p)  at (45-360/4*\p:1.0){};
   \node[mathsc] at (45-360/4*\p:1.27) {\n};
  } 
 \draw (P0)--(P1)--(P2)--(P3)--(P0);
 \end{tikzpicture} 
\quad
 \begin{tikzpicture}[baseline=(current bounding box.center)]
 \foreach \p/\n in 
{
0/{(1,2)},
1/{(3,4)},
2/5,
3/6
 }
 { 
   \node[intvertex] (P\p)  at (45-360/4*\p:1.0){};
   \node[mathsc] at (45-360/4*\p:1.27) {\n};
  } 
 \draw (P0)--(P1)--(P2)--(P3)--(P0);
 \end{tikzpicture} 
\quad
\begin{tikzpicture}[baseline=(current bounding box.center)]
 \foreach \p/\n in 
{
0/1,
1/2,
2/{(3,4)},
3/{(5,6)}
 }
 { 
   \node[intvertex] (P\p)  at (45-360/4*\p:1.0){};
   \node[mathsc] at (45-360/4*\p:1.27) {\n};
  } 
 \draw (P0)--(P1)--(P2)--(P3)--(P0);
 \end{tikzpicture} 
\]

\begin{proof}[Preuve du Théorème \ref{thm:5->n}] On raisonne par récurrence en supposant que
  l'équation \rm{Eq$_5$} implique les relations \rm{Eq$_m$} pour $m<n$. Le cas
  initial de $n=6$ ayant déjà été traité, on suppose ici que $n\geqs 7$.

On suppose d'abord que $n=2k$ avec donc
  $k\geqs 4$. On note $K$ l'ensemble $\{1, \ldots , k\}$ et pour $J\subset K$
  on note $f_{2J} : \m_{0,n}\lra \m_{0,n-|J|}$ le morphisme oubliant les points
  marqués $z_{2j}$ pour $j$ dans $J$.

Pour $J=\{j_1,\ldots,j_l\}$, le polygone $\Gamma_n^{2J}$ associé à $\m_{0,n-|J|}$ et $f_{2J}$ se
représente par 
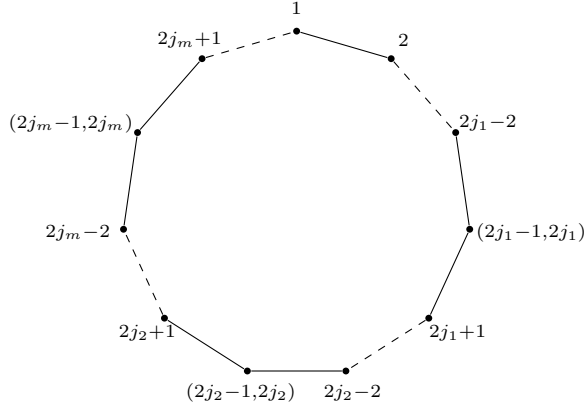
\[
\begin{tikzpicture}[baseline=(current bounding box.center)]
 \foreach \p/\n/\xs in 
{
0/1/0,
1/2/0,
2/{2j_1-2}/1,
3/{(2j_1-1,2j_1)}/{3.4},
4/{2j_1+1}/1,
5/{2j_2-2}/0,
6/{(2j_2-1,2j_2)}/0,
7/{2j_2+1}/0,
8/{2j_m-2}/{-2},
9/{(2j_m-1,2j_m)}/{-4},
10/{2j_m+1}/0
 }
 { 
   \node[intvertex] (P\p)  at (90-360/11*\p:2.3){};
   \node[mathsc,xshift=\xs ex] at (90-360/11*\p:2.6) {\n};
  } 
\draw (P0)--(P1);
\draw[dashed] (P1)--(P2);
\draw (P2)--(P3)--(P4);
\draw[dashed] (P4)--(P5);
\draw (P5)--(P6)--(P7);
\draw[dashed] (P7)--(P8);
\draw (P8)--(P9)--(P10);
\draw[dashed] (P10)--(P0);
 \end{tikzpicture} 
\]  
L'équation fonctionnelle \rm{Eq$_{n-|J|}$} est une conséquence de \rm{Eq$_5$} par
l'induction et induit par composition avec $f_{2J}$ une équation  sur $\m_{0,n}$
\[
\sum_{(\mb i,\mb j)\in \chi_{n-l}}L\left(f_{2J}^*(u_{\mb i,\mb j}^{n,2J})\right))=
\frac{(n-l-3)(n-l-2)}{2}L(1)
\]
avec $l=|J|$. On note $E_{2J}$ le membre de gauche de cette équation et
$E_n=E_{2\emptyset}$ le membre de gauche de  \rm{Eq$_{n}$}:
\[
\sum_{\{i,j\} \in \chi_n} L(u_{i,j})=\frac{(n-3)(n-2)}{2}L(1).
\]
Il nous suffit de montrer que 
\[
E=\sum_{l=0}^k(-1)^l\left(\sum_{\substack{J\subset K \\ |J|=l}}E_{2J} \right)=0.
\]
Cette équation montre que $E_n$ s'écrit comme la somme de termes de type $E_{2J}$
qui sont constants égaux à $\frac{(n-|J|-3)(n-|J|-2)}{2}L(1)$. On obtiendra
ainsi que $E_n$ est constant. La valeur de cette constante étant elle donnée par
spécialisation, on en déduit bien la relation   \rm{Eq$_{n}$}. 

On a déjà remarqué que $f_{2J}^*(u_{\mb i,\mb j}^{n,2J})$ ne dépend pas de $J$
tout en entier mais simplement de la paire $\{\mb i, \mb j\}$ et que
\[
u_{\mb i, \mb j}=f_{2J}^*(u_{\mb i,\mb j}^{n,2J})=\prod_{i \in \mb i, \, j\in
  \mb j}u_{i,j}.
\]  
On peut donc réécrire la somme $E$ en terme des cordes $\{ \mb i, \mb j\}$
possibles :
\[
E=\sum_{\mb i, \mb j }C_{\mb i, \mb j}L(u_{\mb i, \mb j}).
\]
On montre ci après par une étude combinatoire que $C_{\mb i, \mb j}=0$.

Les cordes $\{ \mb i, \mb j\}$ sont de trois types :
\begin{itemize}
\item $\mb i=(2i-1,2i)$ et $\mb j=(2j-1,2j)$ pour $i,j \in K$,
\item $\mb i=(i)=i$ et $\mb j=(2j-1,2j)$ pour $j \in K$ et $i\in \{1,\ldots,n\}$,
\item $\mb i=(i)=i$ et $\mb j=(j)=j$ pour $i,j\in \{1,\ldots,n\}$. 
\end{itemize} 
Pour chaque type, il faut ajouter certaines restrictions afin de ne
travailler qu'avec des cordes strictes. Pour le premier type de corde $\mb
i=(2i-1,2i)$ et $\mb j=(2j-1,2j)$,  $\{\mb i, \mb j\}$ est une corde stricte de
$\Gamma_n^{2J}$ si et seulement si $i,j \in J\subset K$ et $i\neq j-1,j,j+1$ modulo
$k$. La seconde condition caractérise l'existence d'une telle corde. Lorsqu'elle
est satisfaite,  il y a donc $\binom{k-2}{l-2}$ ensembles $J$ possibles avec
$|J|=l\geqs 2$. On trouve ainsi
\[
C_{(2i-1,2i), (2j-1,2j)}=\sum_{l=2}^k(-1)^l\binom{k-2}{l-2}=0.
\]  

Lorsque la corde $\{\mb i, \mb j\}=\{i,j\}$ est du troisième type, on a nécessairement
$l\leqs k-2$ et $i\neq j-1,j,j+1$. C'est une corde de $\Gamma_n^{2J}$ si et
seulement si
\[
\{\frac{i+1}{2}, \frac i 2 , \frac{j+1}{2}, \frac j 2 \}\cap J=\emptyset. 
\]
Deux des éléments de l'ensemble de gauche étant des demi-entiers, la condition
ne porte que sur deux éléments de $K$ qui ne doivent pas être dans $J$. On a donc 
\[
C_{i,j}=\sum_{l=0}^{k-2}(-1)^l\binom{k-2}{l}=0.
\]  
On raisonne de même dans le cas d'une corde de type ``intermédiaire'' $\{\mb i,\mb
  j\}=\{i,(2j-1,2j)\}$ pour obtenir 
\[
C_{i,(2j-1,2j)}=\sum_{l=1}^{k-1}(-1)^l\binom{k-2}{l-1}=0.
\]

Ceci conclut le cas $n=2k$ ($k\geqs 4$). Le cas où $n$ est impair $n=2k+1$
se traite par  la même étude combinatoire avec $K=\{1,\ldots,k\}$. Dans ce cas,
il y a en plus les cordes de type $\{\mb i, \mb j\}=\{2n+1,j\}$ et $\{\mb i, \mb
j\}=\{2n+1,(2j,2j-1)\}$ qui se traitent exactement comme les cordes du troisième
type et du type intermédiaire ci-dessus. 
\end{proof}
On présente ci-dessous les polygones $\Gamma_{n}^{2J}$ présents dans la preuve
pour $n=8$ puis $n=7$. %
\newcommand{\radius}{0.7}
\newcommand{\radiusb}{0.9}
\newcommand{\polyg}[3][90]{
\begin{tikzpicture}[baseline=(current bounding box.center)]
 \foreach \xs/\p/\pp/\n in 
{
#3
 }
 { %
   \coordinate (C\p)  at (#1-360/#2*\p:{\radius});
   \coordinate (C\pp)  at (#1-360/#2*\pp:{\radius});
   \node[intvertex] (P\p)  at (#1-360/#2*\p:{\radius}){};
   \node[mathss,xshift= \xs ex] at (#1-360/#2*\p:{\radiusb}) {\n};
   \draw (C\p)--(C\pp);
  } 
 \end{tikzpicture}
}
Le cas $n=8$, $K=\{1,2,3,4 \}$ :
\begin{gather*}
+\left(
\polyg{8}{
0/0/1/1,
0/1/2/2,
0/2/3/3,
0/3/4/4,
0/4/5/5,
0/5/6/{6},
0/6/7/7,
0/7/0/8
}
\right)
\tag{$|J|=0$}
\\
-\left(
\polyg{7}{
0/0/1/{(1,2)},
0/1/2/3,
0/2/3/4,
0/3/4/5,
0/4/5/6,
0/5/6/{7},
0/6/0/8
}
\polyg{7}{
0/0/1/1,
0/1/2/2,
1/2/3/{(3,4)},
0/3/4/5,
0/4/5/6,
0/5/6/{7},
0/6/0/8
}
\polyg{7}{
0/0/1/1,
0/1/2/2,
0/2/3/3,
0/3/4/4,
0/4/5/{(5,6)},
0/5/6/{7},
0/6/0/8
}
\polyg{7}{
0/0/1/1,
0/1/2/2,
0/2/3/3,
0/3/4/4,
0/4/5/5,
0/5/6/{6},
{-0.65}/6/0/{(7,8)}
}
\right)
\tag{$|J|=1$}
\\
+\left(
\begin{split}
\polyg{6}{
0/0/1/{(1,2)},
1/1/2/{(3,4)},
0/2/3/5,
0/3/4/6,
0/4/5/7,
0/5/0/{8}
}
\polyg{6}{
0/0/1/{(1,2)},
0/1/2/3,
0/2/3/4,
0/3/4/{(5,6)},
0/4/5/7,
0/5/0/{8}
}&
\polyg{6}{
0/0/1/{(1,2)},
0/1/2/3,
0/2/3/4,
0/3/4/5,
0/4/5/6,
{-1}/5/0/{(7,8)}
}
\polyg{6}{
0/0/1/1,
0/1/2/2,
1/2/3/{(3,4)},
0/3/4/{(5,6)},
0/4/5/7,
0/5/0/{8}
}
\\
\polyg{6}{
0/0/1/1,
0/1/2/2,
1/2/3/{(3,4)},
0/3/4/5,
0/4/5/6,
{-1}/5/0/{(7,8)}
}
&
\polyg{6}{
0/0/1/1,
0/1/2/2,
0/2/3/3,
0/3/4/4,
{-1}/4/5/{(5,6)},
{-1}/5/0/{(7,8)}
}
\end{split}
\right)
\tag{$|J|=2$}
\\
-\left(
\polyg{5}{
0/0/1/{(1,2)},
1/1/2/{(3,4)},
0/2/3/{(5,6)},
0/3/4/7,
0/4/0/8
}
\polyg{5}{
0/0/1/{(1,2)},
1/1/2/{(3,4)},
0/2/3/5,
0/3/4/6,
{-1}/4/0/{(7,8)}
}
\polyg{5}{
0/0/1/{(1,2)},
0/1/2/2,
0/2/3/3,
0/3/4/{(4,6)},
{-1}/4/0/{(7,8)}
}
\polyg{5}{
0/0/1/1,
0/1/2/2,
0/2/3/{(3,4)},
0/3/4/{(5,6)},
{-1}/4/0/{(7,8)}
}
\right)
\tag{$|J|=3$}
\\
+\left(
\polyg{4}{
0/0/1/{(1,2)},
1/1/2/{(3,4)},
0/2/3/{(5,6)},
{-1}/3/0/{(7,8)}
}
\right)
\tag{$|J|=4$}
\end{gather*}
Le cas $n=7$, $K=\{1,2,3\}$ :
\begin{gather*}
+\left(
\polyg{7}{
0/0/1/1,
0/1/2/2,
0/2/3/3,
0/3/4/4,
0/4/5/5,
0/5/6/{6},
0/6/0/7
}
\right)
\tag{$|J|=0$}
\\
-\left(
\polyg{6}{
0/0/1/{(1,2)},
0/1/2/3,
0/2/3/4,
0/3/4/5,
0/4/5/6,
0/5/0/{7}
}
\polyg{6}{
0/0/1/1,
0/1/2/2,
1/2/3/{(3,4)},
0/3/4/5,
0/4/5/6,
0/5/0/{7}
}
\polyg{6}{
0/0/1/1,
0/1/2/2,
0/2/3/3,
0/3/4/4,
{-1}/4/5/{(5,6)},
0/5/0/{7}
}
%
\right)
\tag{$|J|=1$}
\\
+\left(
\polyg{5}{
0/0/1/{(1,2)},
1/1/2/{(3,4)},
0/2/3/5,
0/3/4/6,
0/4/0/7
}
\polyg{5}{
0/0/1/{(1,2)},
0/1/2/3,
0/2/3/4,
0/3/4/{(5,6)},
0/4/0/7
}
\polyg{5}{
0/0/1/1,
0/1/2/2,
0/2/3/{(3,4)},
0/3/4/{(5,6)},
0/4/0/7
}
\right)
\tag{$|J|=2$}\\
-\left(
\polyg{4}{
0/0/1/{(1,2)},
1/1/2/{(3,4)},
0/2/3/{(5,6)},
0/3/0/{7}
}
\right)
\tag{$|J|=3$}
\end{gather*}


\section{Quelques remarques d'ouverture}
\`A la connaissance de l'auteur c'est la première fois qu'une famille infinie
d'équations fonctionnelles pour le dilogarithme de Rogers en un nombre croissant de
variables est réduit à la relation à $5$ termes ; l'équation \eqref{eq:Eqn}
étant une équation en $n-3$ variables. Ce travail renforce ainsi l'idée que
l'équation à $5$ termes (avec la relation d'inversion liant $L(x)$ et $L(1/x)$)
est universelle pour le 
dilogarithme : toute équation fonctionnelle du dilogarithme en des fonctions
rationnelles d'un 
nombre quelconque de variables s'y réduit. Le cas d'une variable est lui bien
connu \cite{WotjBSPFE,ZagierRemLi2}.

On remarque par ailleurs que la preuve du Théorème \ref{thm:5->n}, montrant que
la relation à $5$ termes implique les équations \eqref{eq:Eqn}, repose
uniquement sur la combinatoire des espaces de modules de courbes
$\m_{0,n}$. Ainsi, %
si une fonction $F$ satisfait  \eqref{eq:Eqn} pour un certain $n_0\geqs 6$, on
peut espérer montrer par la même méthode que la même fonction $F$ satisfait
\eqref{eq:Eqn} pour tout $n$ assez grand ($n/2\geqs n_0$) après avoir traité un
nombre fini de cas particuliers.
  
En guise de conclusion, on remarque que le travail ci-dessus est aussi valide
dans le cadre du groupe de Bloch $B_2(\F)$ où $\F$ est un corps de nombres. Très
brièvement, le groupe de Bloch $B_2(\F)$ peut se définir (voir par
ex. \cite[Chap.1 \S 4]{ZagierLi2} ou \cite{SuslinICM86,SuslinKthBG}) comme le quotient de 
\[
\ker\left( 
\begin{tikzpicture}[baseline=(current bounding box.center)]
\matrix (m) [matrix of math nodes, row sep=0.5em,
column sep=3.5em, text height=1.5ex, text depth=0.25ex]
{ \Z[\F^*] &  \Lambda^2_{\Z}(\F^*)\otimes_{\Z}\Q \\
{[x]} & {[x] \w [1-x]}  
\\ };
\path[->]
(m-1-1) edge node {} (m-1-2);
\path[|->]
(m-2-1) edge node {} (m-2-2);
\end{tikzpicture}
\right)
\]
par le sous-groupe $Q_2$ engendré par 
\[
[x]+[\frac 1 x ], \quad [x]+[1-x], \quad
[x]+[\frac{1-x}{1-xy}]+[\frac{1-y}{1-xy}]+[y]+[1-xy].
\]
La preuve du Théorème \ref{thm:eqfctn} montre que les combinaisons linéaires de
l'équation \eqref{eq:Eqnf},   qui traduisent les coordonnées diédrales en termes
de fonction rationnelles,
\[
\sum_{\{i,j\} \in \chi_n} [f_{i,j}(u_{1,3},\ldots, u_{1,n-1})]
\]
définissent des éléments du groupe de Bloch $B_2(\F)$. On montrerait de même des
équations fonctionnelles \eqref{eq:Eqn} associées au dilogarithme de Bloch-Wigner. Le
Théorème \ref{thm:5->n} montre que ces éléments sont déjà dans $Q_2$.

Enfin, toujours en relation avec la $K$-théorie des corps de nombres, un
traitement similaire dans le cadre des cycles algébriques et des groupes de
Chow est possible. Le contrôle de l'admissibilité des cycles ainsi que le
comportement de ceux-ci vis à vis des applications d'oubli de points marqués
$\m_{0,n}\lra \m_{0,k}$ nécessite cependant quelques précautions. On développera
cette approche dans un travail ultérieur qui s'intéressera par ailleurs à
prouver d'autres relations classiques des polylogarithmes purement en terme de
cycles algébriques, ceci en s'appuyant sur les constructions 
explicites de \cite{SouMPCC}.

\bibliographystyle{smfalpha}
\bibliography{eqfunctLi2}

\newcommand{\etalchar}[1]{$^{#1}$}
\providecommand{\bysame}{\leavevmode ---\ }
\providecommand{\og}{``}
\providecommand{\fg}{''}
\providecommand{\smfandname}{\&}
\providecommand{\smfedsname}{\'eds.}
\providecommand{\smfedname}{\'ed.}
\providecommand{\smfmastersthesisname}{M\'emoire}
\providecommand{\smfphdthesisname}{Th\`ese}
\begin{thebibliography}{IIK{\etalchar{+}}13b}

\bibitem[Bro09]{BrownMZVPMS}
{\scshape F.~C.~S. Brown} -- {\og Multiple zeta values and periods of moduli
  spaces {$\overline{\mathfrak M}_{0,n}$}\fg}, \emph{Ann. Sci. \'Ec. Norm.
  Sup\'er. (4)} \textbf{42} (2009), no.~3, p.~371--489.

\bibitem[Cha05]{ChapoFIRDCYS}
{\scshape F.~Chapoton} -- {\og Functional identities for the {R}ogers
  dilogarithm associated to cluster {$Y$}-systems\fg}, \emph{Bull. London Math.
  Soc.} \textbf{37} (2005), no.~5, p.~755--760.

\bibitem[IIK{\etalchar{+}}13a]{IIKKN1}
{\scshape R.~Inoue, O.~Iyama, B.~Keller, A.~Kuniba {\normalfont \smfandname}
  T.~Nakanishi} -- {\og Periodicities of {T}-systems and {Y}-systems,
  dilogarithm identities, and cluster algebras {I}: type {$B_r$}\fg},
  \emph{Publ. Res. Inst. Math. Sci.} \textbf{49} (2013), no.~1, p.~1--42.

\bibitem[IIK{\etalchar{+}}13b]{IIKKN2}
\bysame , {\og Periodicities of {T}-systems and {Y}-systems, dilogarithm
  identities, and cluster algebras {II}: types {$C_r$}, {$F_4$}, and
  {$G_2$}\fg}, \emph{Publ. Res. Inst. Math. Sci.} \textbf{49} (2013), no.~1,
  p.~43--85.

\bibitem[Kel11]{KellerCTQDI}
{\scshape B.~Keller} -- {\og On cluster theory and quantum dilogarithm
  identities\fg}, in \emph{Representations of algebras and related topics}, EMS
  Ser. Congr. Rep., Eur. Math. Soc., Z\"urich, 2011, p.~85--116.

\bibitem[Sou12]{SouMPCC}
{\scshape I.~Soud{\`e}res} -- {\og Cycle complex over {$\mathbb P^1$} minus
  {$3$} points : toward multiple zeta values cycles.\fg},
  http://arxiv.org/abs/1210.4653, 2012.

\bibitem[Sus87]{SuslinICM86}
{\scshape A.~A. Suslin} -- {\og Algebraic {$K$}-theory of fields\fg}, in
  \emph{Proceedings of the {I}nternational {C}ongress of {M}athematicians,
  {V}ol. 1, 2 ({B}erkeley, {C}alif., 1986)}, Amer. Math. Soc., Providence, RI,
  1987, p.~222--244.

\bibitem[Sus90]{SuslinKthBG}
\bysame , {\og {$K_3$} of a field, and the {B}loch group\fg}, \emph{Trudy Mat.
  Inst. Steklov.} \textbf{183} (1990), p.~180--199, 229, Translated in Proc.
  Steklov Inst. Math. {{\bf{1}}991}, no. 4, 217--239, Galois theory, rings,
  algebraic groups and their applications (Russian).

\bibitem[Woj91]{WotjBSPFE}
{\scshape Z.~Wojtkowiak} -- {\og The basic structure of polylogarithmic
  functional equations\fg}, in \emph{Structural properties of polylogarithms},
  Math. Surveys Monogr., vol.~37, Amer. Math. Soc., Providence, RI, 1991,
  p.~205--231.

\bibitem[Zag88]{ZagierRemLi2}
{\scshape D.~Zagier} -- {\og The remarkable dilogarithm\fg}, \emph{J. Math.
  Phys. Sci.} \textbf{22} (1988), no.~1, p.~131--145.

\bibitem[Zag89]{ZagierLi2GeomNum}
\bysame , {\og The dilogarithm function in geometry and number theory\fg}, in
  \emph{Number theory and related topics ({B}ombay, 1988)}, Tata Inst. Fund.
  Res. Stud. Math., vol.~12, Tata Inst. Fund. Res., Bombay, 1989, p.~231--249.

\bibitem[Zag07]{ZagierLi2}
{\scshape D.~Zagier} -- {\og The dilogarithm function\fg}, in \emph{Frontiers
  in number theory, physics, and geometry. {II}}, Springer, Berlin, 2007,
  p.~3--65.

\end{thebibliography}
\end{document}